\documentclass{daj}

\dajAUTHORdetails{%
  title = {Square-free Values of Reducible
Polynomials}, 
  author = {Andrew R. Booker, Tim D. Browning},
  plaintextauthor = {Andrew R. Booker, Tim D. Browning},
  }  

\dajEDITORdetails{%
   year={2016},
   number={8},
   received={4 November 2015},   % received date: example: 7 January 2017
   published={1 June 2016},  % published date
   doi={10.19086/da.732},       % XXX = number of paper, e.g. da006 for paper#6
}   %%% END \dajEDITORdetails

\usepackage{amsfonts, amsthm, amsmath, amssymb}

\newtheorem{theorem}{Theorem}[section]
\newtheorem{lemma}[theorem]{Lemma}

\renewcommand{\phi}{\varphi}
\renewcommand{\rho}{\varrho}

\newcommand{\ZZ}{\mathbb{Z}}

\newcommand{\NN}{\mathbb{N}}

\newcommand{\cA}{\mathcal{A}}

\renewcommand{\leq}{\leqslant}
\renewcommand{\le}{\leqslant}
\renewcommand{\geq}{\geqslant}
\renewcommand{\ge}{\geqslant}

\newcommand{\ve}{\varepsilon}

\DeclareMathOperator{\disc}{disc}

\newcommand{\Z}{\mathbb{Z}}
\newcommand{\R}{\mathbb{R}}
\DeclareMathOperator{\Ein}{Ein}

\begin{document}

\begin{frontmatter}[classification=text]
%% EDITOR: this will force the keywords to appear right after the Abstract.
%%   If the abstract is too long and would force the keywords off the
%%   front page, please comment out % [classification=text] above
%%   This way the keywords will be floated on the bottom of the first page
%%   even though the Abstract spills over to the next page.

\title{Square-free Values of Reducible
Polynomials}

\author[ARB]{Andrew R. Booker\thanks{Supported by
EPSRC grant EP/L001454/1}}
\author[TDB]{Tim D. Browning\thanks{Supported by
 ERC grant 306457}}

\begin{abstract}
We calculate  admissible values of $r$ such that a square-free
polynomial with integer coefficients, no fixed prime divisor and
irreducible factors of degree at most $3$ takes infinitely many values
that are a product of at most $r$ distinct primes.
\end{abstract}
\end{frontmatter}

\section{Introduction}

Let $H\in\ZZ[x]$ be a non-constant square-free polynomial of degree $h$ with $\kappa$ irreducible factors.
Unless there is good reason to suppose otherwise, one expects that $H(n)$ 
is square-free infinitely often. In fact, 
assuming that $H$
 has no 
{\em fixed prime divisor},  
one expects that 
$H(n)$ is a product of precisely $\kappa$ distinct primes  for infinitely many values of $n\in \ZZ$. We shall assume that $H$ has no fixed prime divisor in all that follows, by which we mean that there is no prime  $p$
such that   $p\mid H(n)$ for every $n\in \ZZ$. 

When $\kappa=1$ and $h=3$ it follows from work of Erd\H{o}s \cite{sfree} that  $H(n)$ is square-free infinitely often. 
This fact has not yet been extended to a single irreducible quartic polynomial, although work of Granville \cite{G} handles irreducible polynomials of arbitrary degree under the {\em abc conjecture}. 
On the other hand, thanks to 
work of Hooley \cite{hooley'}, 
in the setting of irreducible cubic polynomials we have an asymptotic formula 
for the number of $n\leq x$ for which $H(n)$ is square-free. 
Recent work of Reuss \cite{reuss} has substantially improved  this
asymptotic formula, with the outcome that we now have a power saving in the error term.

The topic of almost-prime values of polynomials has been studied
extensively by several authors.  The  {\em DHR weighted sieve}
of Diamond, Halberstam and Richert (see \cite{dh-applications,
dh-book}) gives an efficient means of producing values of $r$ such that
$\Omega(H(n))\leq r$ for infinitely many $n\in \ZZ$, where $\Omega(m)$
is the total number of prime factors of $m$ counted with multiplicity.
The following result gives  sufficient conditions on polynomials $H$,
not necessarily irreducible, under which $H(n)$ takes infinitely many
values that are simultaneously square-free and have at most $r$ prime
factors, for relatively small values of $r$.

\begin{theorem}\label{thm:main}
Assume that  
$H$ has no fixed prime divisor and that 
each irreducible factor of $H$ has degree at most $3$.
Then there exists $r\in \NN$ such that 
$$
\sum_{\substack{n\leq x\\ \Omega(H(n))\leq r}} \mu^2(H(n)) \gg \frac{x}{(\log x)^\kappa}.
$$
Admissible values of $r$ are recorded in Table~\ref{tab:rmin} for $h\le20$
and at \cite{sourcecode} for $h\le50$.
\end{theorem}

\begin{table}[h!]
\begin{tabular}{|r|*{20}{c}|}\hline
$h\setminus\kappa$
&1&2&3&4&5&6&7&8&9&10&11&12&13&14&15&16&17&18&19&20\\\hline
1&1&&&&&&&&&&&&&&&&&&&\\
2&2&5&&&&&&&&&&&&&&&&&&\\
3&4&6&8&&&&&&&&&&&&&&&&&\\
4&&7&10&12&&&&&&&&&&&&&&&&\\
5&&8&11&13&16&&&&&&&&&&&&&&&\\
6&&9&12&15&18&20&&&&&&&&&&&&&&\\
7&&&13&16&19&22&25&&&&&&&&&&&&&\\
8&&&14&17&20&23&26&29&&&&&&&&&&&&\\
9&&&16&19&22&25&28&31&34&&&&&&&&&&&\\
10&&&&20&23&26&29&33&36&39&&&&&&&&&&\\
11&&&&21&24&27&31&34&37&41&44&&&&&&&&&\\
12&&&&22&25&29&32&35&39&42&45&49&&&&&&&&\\
13&&&&&27&30&33&37&40&44&47&51&54&&&&&&&\\
14&&&&&28&31&35&38&42&45&49&52&56&59&&&&&&\\
15&&&&&29&33&36&40&43&47&50&54&57&61&64&&&&&\\
16&&&&&&34&37&41&45&48&52&55&59&62&66&70&&&&\\
17&&&&&&35&39&42&46&50&53&57&60&64&68&71&75&&&\\
18&&&&&&36&40&44&47&51&55&58&62&66&69&73&77&80&&\\
19&&&&&&&41&45&49&52&56&60&63&67&71&75&78&82&86&\\
20&&&&&&&42&46&50&54&57&61&65&69&73&76&80&84&88&92\\\hline
\end{tabular}
\caption{Admissible values of $r$ for $h\le20$.\label{tab:rmin}}
\end{table}

One of our motivations for this work is the following application
derived in \cite{euclid}.  Given any finite set of prime numbers, Euclid
showed how to obtain another prime not in the set. Applied iteratively
beginning with a single prime, one obtains an infinite sequence of
distinct prime numbers. The set of
all possible sequences obtained from a given starting prime $p$ has the
natural structure of a directed graph, say $G_p$, which one can then ask
questions about; e.g., is $G_p$ a tree?  Applying Theorem~\ref{thm:main}
to the polynomial 
$$
H=(x^2+x+1)(x^2+1)(x^3+x^2+2x+1),
$$
it was shown in \cite{euclid} that $G_p$ is not a tree for a positive
proportion of $p$, and in fact there are infinitely many coprime pairs
$(a_i,q_i)\in\NN^2$, with $(q_i,q_j)=1$ for $i\ne j$, such that $G_p$
has a loop of height at most $13$ for any $p\equiv a_i\;(\text{mod }q_i)$.

For any $z\geq 1$ let
$
P(z)=\prod_{p\leq z}p
$
and let $\delta>0$ be a small parameter to be chosen in due course. 
Our starting point in the proof of Theorem \ref{thm:main}
is the observation that 
$$
\sum_{\substack{n\leq x\\ \Omega(H(n))\leq r}} \mu^2(H(n))\geq 
\sum_{\substack{n\leq x\\ \Omega(H(n))\leq r\\ (H(n),P(x^\delta))=1}} 
\mu^2(H(n))\geq \Sigma_1-\Sigma_2,
$$
where 
$$
\Sigma_1=\#\left\{n\leq x:  \Omega(H(n))\leq r, ~(H(n),P(x^\delta))=1\right\}
$$
and 
$$
\Sigma_2=
\#\left\{n\leq x: \exists~p>x^\delta \text{ s.t. } p^2\mid H(n)\right\}.
$$
Our strategy is now evident. First we will  seek a lower bound for $\Sigma_1$, for which we shall invoke 
the DHR weighted sieve \cite{dh-applications} in  \S \ref{s:sieve}. 
Next, in 
\S \ref{s:reuss}, we shall prove that there exists $\eta>0$ such that 
 \begin{equation}\label{eq:sigma2}
 \Sigma_2=O(x^{1-\eta}),
\end{equation} 
provided that each irreducible factor of $H$ has degree at most $3$.
The key input here will come from the work of Reuss \cite{reuss} that we have already mentioned and which is based on the {\em approximate determinant method}.
Finally, in \S \ref{s:andy}, we shall turn to the question of optimising the value of $r$ for which the lower bound in Theorem \ref{thm:main} holds.

A well-known feature of  the weighted sieve is that any admissible
value of $r$ can only be derived through a precise analysis of  certain
integrals involving complicated functions $\sigma, f,F$  that arise as
solutions to certain differential-delay problems. These are all described
in Theorem \ref{thm:dhr} below.  Wheeler, in his thesis \cite{wheeler},
gave algorithms for solving the differential-delay problem via Gaussian
quadrature with rigorous error bounds, and some of the theoretical
work concerning the solution of this problem (see \cite{dh-book}
for a comprehensive treatment) relies on his computations.  However,
\cite{wheeler} is rather difficult to obtain and there appears to have
been no published account of Wheeler's algorithms or source code. With
this in mind, in \S\ref{s:andy} we describe an independent implementation
based on polynomial approximations. Thanks to the exponential gains in
computer power since Wheeler's thesis was written, we can afford to leave
more of the work to the computer, to the point that our error estimates
require nothing more complicated than the alternating series test or
the Lagrange form of the remainder term in Taylor's theorem.  We rely
heavily throughout on Fredrik Johansson's excellent library {\tt Arb}
\cite{arb} for arbitrary-precision interval arithmetic. Thus, we make
no assumptions about round-off error, so that, modulo bugs in the code
or computer hardware, our results are rigorous.  The end result of our
computations is the online table \cite{sourcecode} of minimum values of
$r$ for every pair $\kappa$, $h$ with $\kappa\le h\le50$, from which
Table~\ref{tab:rmin} is derived.  The reader wishing to compute $r$
for numbers outside of the tables is welcome to make use of our program,
also available at \cite{sourcecode}.

\medskip

We conclude the introduction by proving one more result used in
\cite{euclid}.  Namely, if we drop the condition $\Omega(H(n))\leq r$
then we can promote the lower bound in Theorem \ref{thm:main} to an
asymptotic formula by adapting work of Hooley \cite[Chapter~4]{hooley}
and invoking Reuss' work. To do so we write
$$
\sum_{\substack{n\leq x}} \mu^2(H(n)) =
N(x)+O(E(x)+\Sigma_2),
$$
where $\Sigma_2$ is as above and 
$N(x)$ (resp.\ $E(x)$) denotes the 
number of $n\leq x$ for which $H(n)$ is not divisible by $p^2$ for any prime $p\leq \frac{1}{6}\log x$ (resp.\ there exists a prime $p\in (\frac{1}{6}\log x,x^\delta]$ such that  $p^2\mid H(n)$).
Combining the estimate  \cite[Eq.~(125)]{hooley} for $N(x)$, with \eqref{eq:sigma2} and the estimate \cite[Eq. ~(127)]{hooley} for $E(x)$, we arrive at the following result. 

\begin{theorem}
Assume that  each irreducible factor of $H$ has degree at most $3$ and let
\begin{equation}\label{eq:rho}
\rho(q)=\#\{\nu\bmod{q}: H(\nu)\equiv 0\bmod{q}\},
\end{equation}
for any $q\in \NN$.
Then 
$$
\sum_{\substack{n\leq x}} \mu^2(H(n))  =x \prod_p
\left(1-\frac{\rho(p^2)}{p^2}\right) +O\!\left(\frac{x}{\log x}\right).
$$
\end{theorem}

\section{Polynomial congruences}
An important function in our work is the multiplicative arithmetic function $\rho(q)$ that was defined in \eqref{eq:rho}.
Our assumption that $H$ has no fixed prime divisor is equivalent to the statement that $\rho(p)<p$ for every prime $p$.
We henceforth assume that 
$$
H(x)=H_1(x)\cdots H_\kappa(x)
$$ 
for a system  $H_1,\dots,H_\kappa\in \ZZ[x]$ of pairwise non-proportional irreducible polynomials, with  
$d_i=\deg(H_i)$ for $1\leq i\leq \kappa$. 

It will pay dividends to consider the general
function 
$$\rho(G;q)=\#\{\nu\bmod{q}: G(\nu)\equiv 0\bmod{q}\},$$ 
associated to  any irreducible polynomial $G\in \ZZ[x]$ of degree $d$
and any $q\in \NN$. Note that 
$\rho(G;q)$ is a multiplicative function of $q$ and 
$\rho(q)=\rho(H;q)$.
Assuming that $p$ does not divide the content of $G$, we
 have the basic estimate
\begin{equation}\label{eq:upper}
\rho(G;p^k)\ll \begin{cases}
1, & \text{if $p\nmid \disc(h)$},\\
p^{(1-1/d)k}, &\text{otherwise,}
\end{cases}
\end{equation}
where the implied constant is only allowed to depend on $d$.
Here the first bound is a trivial application of Hensel's lemma, combined with the inequality
$\rho(G;p)\leq d$. The second estimate is well known (see \cite[Lemma 2]{nair}, for example).

Returning to  $\rho(p^k)$ for a prime $p$ and $k\in \NN$, 
let 
\begin{equation}\label{eq:Delta} 
\Delta=6\prod_i |\mathrm{disc}(H_i)|\prod_{i\neq j}|\mathrm{Res} (H_i,H_j)|,
\end{equation}
where $\mathrm{Res} (H_i,H_j)$ is the resultant of $H_i$ and $H_j$. 
This an integer which vanishes if and only if $H_i$ and $H_j$ share a common root. 
Our assumptions on the system $H_1,\dots,H_\kappa$  imply that $\Delta$ is a positive integer and we henceforth allow all of our implied constants to depend on it. Thus we clearly have $\rho(p^k)=O_k(1)$ if  $p\mid \Delta$, where the implied constant depends on $k$. If
 $p\nmid \Delta$, then $p$ doesn't divide any of the resultants $\mathrm{Res}(H_i,H_j)$, and so 
$$
\rho(p^k)\leq \sum_{i=1}^\kappa  \rho(H_i,p^k)
\ll 1,
$$
by \eqref{eq:upper}. Hence
it follows that 
\begin{equation}\label{eq:celest}
\rho(p^k)=O_k(1),
\end{equation}
for any prime $p$ and any $k\in \NN$.

\section{The DHR weighted sieve}\label{s:sieve}

We begin by recalling the set-up for a version of the DHR weighted sieve \cite{dh-applications} that we shall use.
Let $\cA$ be a finite sequence of integers. 
Under suitable hypotheses, the DHR weighted sieve produces a lower bound 
for the cardinality
$$
\#\{a\in \cA: \Omega(a)\leq r, (a,P(z))=1\},
$$
for suitable $r$ and $z$. In fact, in the literature, it is usually recorded as a lower bound
for $\#\{a\in \cA: \Omega(a)\leq r\}$, but the results are actually valid for the restricted cardinality in which any $a\in \cA$ is forbidden to have prime factors less than $z$. (See the proof of \cite[Thm.~11.1]{dh-book} for elucidation of this point.)
In order to get a lower bound for $\Sigma_1$ we take 
$$
\cA=\{H(n): n\leq x\} 
$$
and $z=x^{1/v}$, for a  parameter $v>1$.

When $\kappa=1$, one can solve the differential-delay problem explicitly,
and this leads to the conclusion that $r=h+1$ is admissible for every
$h\ge1$ (see \cite{richert}).  Moreover, for $\kappa=1$ and $h\le 2$ we may
take $r=h$ by the prime number theorem for arithmetic progressions and
\cite{lemke-oliver}.  These values are reflected in Table~\ref{tab:rmin}.

For $\kappa\ge2$ we refer to \cite{dh-applications}. One sees that 
we have a sieve of dimension $\kappa$, that 
all of the hypotheses of \cite[Thm~1]{dh-applications} 
are met with 
$X=x$, $\omega(d)=\rho(d)$, $\tau=1$ and that the choice $\mu=h$ is acceptable.
%T [the sieve hypotheses]
%Let us now detail the various hypotheses that enter into the DHR weighted sieve for a general sequence $\cA$, before returning to see what happens for \eqref{eq:our-A}.
%Denote by $\cP^c$ the complement of $\cP$ with respect to the set of all primes. 
%Suppose that there exists an approximation $X$ to $\#\cA$ and a non-negative multiplicative arithmetic function 
%$\omega$ satisfying 
%$$
%\omega(1)=1, \quad \text{$\omega(p)=0$ if $p\in \cP^c$}, \quad
%\text{$\omega(p)\in [0,p)$ if $p\in \cP$},
%$$
%and for some constants $\kappa>1$, $A\geq 2$, 
%$$
%\prod_{z_1\leq p\leq z} \left(1-\frac{\omega(p)}{p}\right)^{-1}\leq \left(\frac{\log z}{\log z_1}\right)^\kappa \left(1+\frac{A}{\log z_1}\right)\quad \text{if $2\leq z_1<z$},
%$$
%such that 
%$$
%\sum_{\substack{d<X^\tau (\log X)^{-A_1} \\ (d,\cP^c)=1}} \mu^2(d) 4^{\nu(d)}|R_d| \leq A_2 \frac{X}{(\log X)^{\kappa+1}},
%$$
%for some constants $\tau\in (0,1]$, $A_1\geq 1$ and $A_2\geq 2$. Here 
%$\nu(d)$ denotes the number of distinct prime factors of $d$ and 
%$R_d$ is the remainder function
%$$
%R_d=\#\left\{a\in \cA: a\equiv 0\bmod{d}\right\}-\frac{\omega(d)}{d}X.
%$$
%Let $\mu$ be a constant such that $\max_{a\in \cA}|a|\leq X^{\tau \mu}$.
For any $v,w\in\R_{>0}$ satisfying $v>w+1$,
it therefore follows  that
$$
\#\{n\leq x:\Omega(H(n))\le r, (H(n), P(x^{1/v}) )=1\}
\gg\frac{x}{(\log{x})^\kappa},
$$
provided that $r\ge\lfloor R(v,w)\rfloor$, where
\begin{equation}\label{eqn:Rdef}
R(v,w)=\frac{hv}{v-w}+\frac{\kappa}{f(v)}\int_w^{v-1}F(u)
\big((v-u)^{-1}-(v-w)^{-1}\big)\,du
\end{equation}
and $f,F$ is a pair of solutions to
the following differential-delay problem:
\begin{theorem}[Diamond, Halberstam and Richert \cite{dhr1,dhr2,dhr3}]
\label{thm:dhr}
Let $\kappa>1$ be a real number, and let
$\sigma:\R_{>0}\to\R$ be the continuous solution of the system
\begin{equation}\label{eqn:sigmadef}
\begin{aligned}
u^{-\kappa}\sigma(u)&=\frac{(2e^\gamma)^{-\kappa}}{\Gamma(1+\kappa)}
&\mbox{for }u\in(0,2],\\
\frac{d}{du}\bigl(u^{-\kappa}\sigma(u)\bigr)
&=-\kappa u^{-\kappa-1}\sigma(u-2)
&\mbox{for }u>2.
\end{aligned}
\end{equation}
Then there are real numbers $\alpha>\beta>2$ such that the system
\begin{equation}\label{eqn:Ffdef}
\begin{aligned}
F(u)&=\frac1{\sigma(u)}&\mbox{for }u\in(0,\alpha],\\
f(u)&=0&\mbox{for }u\in(0,\beta],\\
\frac{d}{du}\bigl(u^\kappa F(u)\bigr)&=\kappa u^{\kappa-1}f(u-1)
&\mbox{for }u>\alpha,\\
\frac{d}{du}\bigl(u^\kappa f(u)\bigr)&=\kappa u^{\kappa-1}F(u-1)
&\mbox{for }u>\beta
\end{aligned}
\end{equation}
has continuous solutions $F,f:\R_{>0}\to\R$ such that
$F(u)$ decreases monotonically, $f(u)$ increases monotonically, and
$$
F(u)=1+O(e^{-u}),\quad
f(u)=1+O(e^{-u}).
$$
\end{theorem}

Now it is clear that 
$$
\Sigma_1\geq \#\{n\leq x:\Omega(H(n))\le r, (H(n), P(x^{1/v}) )=1\}
$$
if $\delta<1/v$.  
We therefore conclude that 
$\Sigma_1\gg x/(\log x)^\kappa$ provided that $r\geq \lfloor R(v,w)\rfloor$ and $\delta<1/v$, which is plainly satisfactory for Theorem \ref{thm:main}.

\section{Large square divisors}\label{s:reuss}

The goal of this section is to prove the upper bound \eqref{eq:sigma2}  for $\Sigma_2$.
Our treatment is modelled on an  argument of  Hooley \cite[Chapter~4]{hooley}.
We partition the interval $(x^\delta,\infty)$ into $I_1\cup I_2\cup I_3$, where
$$
I_1=(x^\delta, x^{1-\delta}], \quad 
I_2=(x^{1-\delta}, x^{1+\delta}], \quad 
I_3=(x^{1+\delta}, \infty).
$$
We clearly have 
$$
\Sigma_2\leq  \Sigma_{2,1}+\Sigma_{2,2}+\Sigma_{2,3},
$$
where $\Sigma_{2,i}$ is the number of $n\leq x$ such that $p^2\mid H(n)$ for some  prime $p\in I_i$.

We now make the further assumption that $\delta<1/11$.
Once taken together, 
the following three results suffice to establish \eqref{eq:sigma2}, as required to complete the proof of  Theorem~\ref{thm:main}.

\begin{lemma}
We have $\Sigma_{2,1}=O(x^{1-\delta})$.
\end{lemma}

\begin{proof}
Breaking the sum into residue classes modulo $p^2$, we find that 
$$
\Sigma_{2,1}\leq \sum_{p\in I_1} \sum_{\substack{\nu \bmod{p^2}\\ H(\nu)\equiv 0\bmod{p^2}}} 
\#\{n\leq x: n\equiv \nu \bmod{p^2}\}.
$$
The inner cardinality is $x/p^2+O(1)$. Hence
$$
\Sigma_{2,1}\ll \sum_{p\in I_1} \left(\frac{x}{p^2}+ 1\right) \ll x^{1-\delta},
$$
since $\rho(p^2)=O(1)$ by \eqref{eq:celest}. This establishes the result.
\end{proof}

 \begin{lemma}
We have $\Sigma_{2,3}=O(x^{1-\delta+\ve})$, for any $\ve>0$. 
\end{lemma}

\begin{proof}
The contribution to $\Sigma_{2,3}$ from $n\leq x$ such that $H(n)=0$ is $O(1)$, which is satisfactory. Hence we may focus on $n$ for which  $H(n)\neq 0$.
We assume that $x$ is large so that $p\nmid \Delta$, whenever $p>x^{1+\delta}$, 
where $\Delta$ is given by \eqref{eq:Delta}. 
 But then $p^2\mid H(n)$ implies that  
$p^2\mid H_i(n)$ for some $i\in \{1,\dots,\kappa\}$, 
since $p$ doesn't divide any of the resultants $\mathrm{Res}(H_i,H_j)$. 
In particular only the irreducible factors of degree $3$ occur, since 
$x^{d_i} \gg |H_i(n)|\geq p^2 \gg x^{2+2\delta}$.
It follows that
\begin{align*}
\Sigma_{2,3}(x)
&\leq \sum_{\substack{1\leq i\leq \kappa\\ d_i=3}}\#\{(n,p,m): H_i(n)=p^2m, ~n\leq x, ~p>x^{1+\delta}, ~m\ll x^{1-2\delta}\}\\
&\leq \sum_{\substack{1\leq i\leq \kappa\\ d_i=3}}\sum_{m\ll x^{1-2\delta}} \#\{n\leq x: H_i(n)\equiv 0 \bmod{m}, ~H_i(n)/m=\square\}.
\end{align*}
Breaking into residue classes modulo $m$ the inner cardinality is seen to be
$$
\sum_{\substack{\nu\bmod{m}\\ H_i(\nu)\equiv0 \bmod{m}}}
\#\{n\leq x: n\equiv \nu\bmod{m}, ~H_i(n)/m=\square\}.
$$
We make the change of variables $n=\nu+mu$ for $|u|\leq x/m+1$.

At this point we call upon work of Heath-Brown 
\cite[Thm.~15]{cime} to estimate the inner cardinality. 
Given $\ve>0$  and an absolutely irreducible polynomial $F\in \ZZ[u,v]$ of degree $d$,  
this shows that  there are at most
$$
\ll  (UV)^\ve
\exp\left( \frac{\log U\log V}{\log T}\right)
$$
choices of $(u,v)\in \ZZ^2$ such that $|u|\leq U$, $|v|\leq V$ and $F(u,v)=0$. 
Here $T$ is defined to be the maximum of $U^{e_1}V^{e_2}$, taken over all monomials $u^{e_1}v^{e_2}$ which appear in $F(u,v)$ with non-zero coefficient.  
Moreover,  the implied constant is only allowed to depend on $d$ and $\ve$.
We shall apply this bound with $U=x/m+1$ and 
$F(u,v)=H_i(\nu+mu)-mv^2$. In particular we may take $T\geq V^2$. Thus it follows that 
\begin{align*}
\Sigma_{2,3}(x)
&\ll \sum_{\substack{1\leq i\leq \kappa\\ d_i=3}}\sum_{m\ll x^{1-2\delta}} 
\rho(H_i,m)
\left(\frac{x}{m}\right)^{1/2+\ve},
\end{align*}
for any $\ve>0$.
We now factorise $m=st$ where $s$ is cube-free and $t$ is cube-full, with $(s,t)=1$. Then 
\begin{align*}
\Sigma_{2,3}(x)
&\ll x^{1/2+\ve}\sum_{\substack{1\leq i\leq \kappa\\ d_i=3}}\sum_{\substack{t\ll x^{1-2\delta}\\ \text{$t$  cube-full}}} 
\frac{\rho(H_i,t)}{\sqrt{t}}
\sum_{\substack{s\ll x^{1-2\delta}/t\\ \text{$s$  cube-free}}} 
\frac{\rho(H_i,s)}{\sqrt{s}}.
\end{align*}
Now it follows from \eqref{eq:upper} that 
$\rho(H_i,s)\ll s^\ve$ and 
$\rho(H_i,t)\ll t^{2/3+\ve}$. Hence 
\begin{align*}
\Sigma_{2,3}(x)
&\ll x^{1/2+\ve}\sum_{\substack{t\ll x^{1-2\delta}\\ \text{$t$  cube-full}}} 
\frac{t^{2/3+\ve}}{\sqrt{t}}
\left(\frac{x^{1-2\delta}}{t}\right)^{1/2+\ve}\ll x^{1-\delta/2+3\ve},
\end{align*}
since 
$$
\sum_{\substack{t\leq T \\ \text{$t$  cube-full}}} \frac{1}{t^{1/3}}\ll T^\ve.
$$
We complete the proof of the lemma on redefining the choice of $\ve$.
\end{proof}

 \begin{lemma}
We have  $\Sigma_{2,2}=O(x^{10/11+\delta})$.
\end{lemma}

\begin{proof}
As previously, if  $x$ is large enough we have  $p\nmid \Delta$ whenever $p\in I_2$. 
Thus 
\begin{align*}
\Sigma_{2,2}(x)
&\leq \sum_{\substack{1\leq i\leq \kappa}}
\#\left\{n\leq x: \exists~p\in I_2 \text{ s.t. } p^2\mid H_i(n)\right\}=\sum_{1\leq i\leq \kappa} \Sigma_{2,2}^{(i)},
\end{align*}
say. 
Suppose first that 
$i\in \{1,\dots,\kappa\}$ is such that 
$d_i=\deg(H_i)\leq 2$. Then we can assume that $H_i(n)=p^2m$ with $m\ll x^2/x^{2(1-\delta)}=x^{2\delta}$.
Arguing as in the proof of the previous lemma we find that 
$$
\Sigma_{2,2}^{(i)}\ll \sum_{m\ll x^{2\delta}} \rho(H_i,m)\left(\frac{x}{m}\right)^{1/2+\ve}.
$$
Breaking $m$ into its  cube-free and cube-full part, we are easily led to the 
satisfactory contribution $\Sigma_{2,2}^{(i)}=O(x^{1/2+\delta+2\ve})$.

Now suppose that $i\in \{1,\dots,\kappa\}$ is such that $d_i=3$. 
To estimate $\Sigma_{2,2}^{(i)}$ we shall call upon work of Reuss \cite{reuss}.
Thus we have  the overall contribution 
$$
\Sigma_{2,2}^{(i)}\ll x^\ve \max_{X\ll x}\max_{\substack{X^{1-\delta}\ll A\ll  X^{1+\delta}\\
X^{1-2\delta}\ll B\ll  X^{1+2\delta}}} \mathcal{N}(X;A,B),
$$
where
$$
\mathcal{N}(X;A,B)
=
\#\left\{(n,a,b): 
n\sim X, ~a\sim A, ~b\sim B, ~\mu^2(a)=1, ~H_i(n)=a^2b
\right\}
$$
in the notation of \cite[Lemma 3]{reuss}.
Appealing to \cite[page 283]{reuss} with $d=3$ one finds that 
$
\mathcal{N}(X;A,B)\ll M^{2/3}X^{2/3+\ve},
$
where $M$ is any choice of parameter satisfying the condition 
$$
(AB)^{3/11}X^{-2/11}\ll M\ll \min \{X^{1/2},A^{1/2}\}
$$
from \cite[Lemma 7]{reuss}. A moment's thought shows that $M=X^{4/11+\delta}$ is admissible, which leads to the bound 
$\mathcal{N}(X;A,B)\ll X^{10/11+2\delta/3+\ve}$. 
We conclude the proof of the lemma by 
inserting this into our bound for $\Sigma_{2,2}^{(i)}$ and taking $\ve$ sufficiently small in terms of $\delta$.
\end{proof}

\section{Numerical estimates}\label{s:andy}
In this section, we describe how to compute rigorous numerical
approximations of \eqref{eqn:Rdef}.  
Since our main application requires
only integer values of $\kappa\ge2$, we make this restriction for
convenience. However, our approach should work just as well for arbitrary
real $\kappa\ge1$, and in fact one could likely treat $\kappa$ as a
variable. This might be useful, for instance, for extending the results
of \cite{dh-parameters} to small $\kappa$. (On the other hand, recent
work of Franze \cite{franze} and Blight \cite{blight} has shown that
lower-bound sieves that are superior to the DHR sieve are possible once
$\kappa\ge3$, so such an extension is likely to be of limited interest.)

\subsection{Generalised polynomial representations of continuous
functions}
\label{sec:powerseries}
We begin with some generalities on representing continuous functions
by (generalised) polynomials before getting into the details of
computing \eqref{eqn:Rdef} below.

Let $f:[0,1]\to\R$ be a continuous function.
For a fixed $N\in\Z_{\ge0}$,
we consider representations of $f$ of the form
\begin{equation}\label{eqn:fseries}
f(z)=\sum_{n=0}^Nf_n(z)z^n,
\end{equation}
where each $f_n$ is a bounded function on $[0,1]$. For instance, if $f$ is
smooth then there is such a representation with
$f_n=\frac{f^{(n)}(0)}{n!}$ for $n<N$ and $f_N$ smooth.
We represent each coefficient $f_n$ by an interval with rational endpoints
containing the image $f_n([0,1])$.

Given an interval $[a,b]$, let $\Theta([a,b])$ denote any function
(possibly different at each occurrence) whose image lies in $[a,b]$.
The algebraic and analytic operations on series of the form
\eqref{eqn:fseries} are implemented in a straightforward manner. We
mention a few in particular:
\begin{itemize}
\item If
$f(z)=\sum_{n=0}^Nf_n(z)z^n$
and
$g(z)=\sum_{n=0}^Ng_n(z)z^n$
then
$$
f(z)g(z)=\sum_{n=0}^Nh_n(z)z^n,
$$
where
$h_n(z)=\sum_{i+j=n}f_i(z)g_j(z)$ for $n<N$ and
$$
h_N(z)=\sum_{i+j\ge N}f_i(z)g_j(z)z^{i+j-N}.
$$
We obtain a bounding interval for $h_N$ by evaluating the above in
interval arithmetic with $z$ replaced by $\Theta([0,1])$.
\item
If $f$ is known to be bounded away from $0$ then we may compute a
representation for its reciprocal via
\begin{equation}\label{eqn:reciprocal}
\frac1{f(z)}=\sum_{n=0}^Nk_n(z)z^n,
\end{equation}
where $k_n$ is defined for $n<N$ by the relation
$$
\sum_{n=0}^{N-1}f_nz^n\cdot\sum_{n=0}^{N-1}k_nz^n
\equiv1\pmod{z^N},
$$
and
$$
k_N(z)=-\frac1{f(z)}\sum_{\substack{i\le N,j<N\\i+j\ge N}}
f_i(z)k_j(z)z^{i+j-N}.
$$
To compute a bounding interval for $k_N$ we again use interval arithmetic,
replacing $f(z)$ in the denominator by a known bounding interval.
\item
We have
$$
\exp f(z)=e^{f_0(z)}\prod_{n=1}^N
\left(\sum_{j=0}^{\lceil{N/n}\rceil-1}\frac{f_n(z)^j}{j!}z^{nj}+R_{n,N}(z)z^N\right),
$$
where
$$
R_{n,N}(z)=
\sum_{j=\lceil{N/n}\rceil}^\infty\frac{f_n(z)^jz^{nj-N}}{j!}.
$$
Using the Lagrange form of the remainder term in Taylor's theorem, we
have
\begin{align*}
R_{n,N}(z)&=\frac{f_n(z)^{\lceil{N/n}\rceil}}
{\lceil{N/n}\rceil!}\exp(\Theta([0,1])f_n(z)z^n)z^{n\lceil{N/n}\rceil-N}\\
&=\Theta([0,1])\frac{f_n(z)^{\lceil{N/n}\rceil}}
{\lceil{N/n}\rceil!}\exp(\Theta([0,1])f_n(z)).
\end{align*}
Given a bounding interval for
$f_n$, this yields one for $R_{n,N}$ via interval arithmetic.
\item
Let $I_n$ be a bounding interval for $f_n$. Then we have
\begin{align*}
\int_0^z f(x)\,dx=\sum_{n=0}^N\int_0^zf_n(x)x^n\,dx
&=\sum_{n=0}^N\frac{\Theta(I_n)}{n+1}z^{n+1}\\
&=\sum_{n=1}^N\frac{\Theta(I_{n-1})}{n}z^n
+\Theta([0,1])\frac{\Theta(I_N)}{N+1}z^N.
\end{align*}
\end{itemize}
We conclude with two useful examples employing the Lagrange form of
the remainder term. First, if $\nu\in\R_{>0}$ then
$$
(1+x)^{-\nu}=\sum_{n=0}^{N-1}(-1)^n{{n+\nu-1}\choose{n}}x^n
+(-1)^N{{N+\nu-1}\choose{N}}\Theta([0,1])x^N
\quad\mbox{for }x\in\R_{\ge0}.
$$
Similarly,
$$
e^x=\sum_{n=0}^{N-1}\frac{x^n}{n!}+\frac{e^{\Theta([0,1])x}x^N}{N!}
\quad\mbox{for }x\in\R.
$$

\subsection{Solving the differential-delay system}
Next we describe how to compute the solution to
Theorem~\ref{thm:dhr} for integral $\kappa\ge2$,
beginning with $\alpha$ and $\beta$.
These are found via a number of auxiliary functions. First, we have
$p,q:\R_{>0}\to\R$ defined by the differential-delay equations
$$
\frac{d}{du}\bigl(up(u)\bigr)=\kappa\bigl(p(u)-p(u+1)\bigr),
\quad\lim_{u\to\infty}up(u)=1
$$
and
\begin{equation}\label{eqn:qdef}
\frac{d}{du}\bigl(uq(u)\bigr)=\kappa\bigl(q(u)+q(u+1)\bigr),
\quad\lim_{u\to\infty}u^{1-2\kappa}q(u)=1.
\end{equation}
Second, $\widetilde{\Pi},\widetilde{\Xi}:(2,\infty)\to\R$ are defined by
\begin{equation}\label{eqn:pixidef}
\begin{aligned}
\widetilde{\Pi}(u)&=\frac{up(u)}{\sigma(u)}
+\kappa\int_{u-2}^u\frac{p(t+1)}{\sigma(t)}\,dt,\\
\widetilde{\Xi}(u)&=\frac{uq(u)}{\sigma(u)}
-\kappa\int_{u-2}^u\frac{q(t+1)}{\sigma(t)}\,dt.
\end{aligned}
\end{equation}
Then for $\alpha$, $\beta$ we may take any solution to the pair of
equations
\begin{equation}\label{eqn:alphabetadef}
\begin{aligned}
\widetilde{\Pi}(\alpha)+\kappa(\alpha-1)^{1-\kappa}p(\alpha-1)
\int_{\beta}^{\alpha-1}\frac{t^{\kappa-1}}{\sigma(t-1)}\,dt&=2,\\
\widetilde{\Xi}(\alpha)-\kappa(\alpha-1)^{1-\kappa}q(\alpha-1)
\int_{\beta}^{\alpha-1}\frac{t^{\kappa-1}}{\sigma(t-1)}\,dt&=0.
\end{aligned}
\end{equation}
(Conjecturally, the solution is unique.)
We describe the computation of each piece in turn.

\subsubsection{Computing $p$}
By \cite[(12.11)]{dh-book}, we have the following integral representation for
$p$:
\begin{equation}\label{eqn:pdiffeq}
p(u)=\int_0^\infty\exp\big(-\kappa\Ein(x)-ux\big)\,dx,
\end{equation}
where
\begin{equation}\label{eqn:Edef}
\Ein(x)=\int_0^x\frac{1-e^{-t}}t\,dt
=\sum_{n=1}^\infty\frac{(-1)^{n-1}}{n\cdot n!}x^n.
\end{equation}
We split the integral \eqref{eqn:pdiffeq} as
\begin{align*}
&\sum_{m=0}^{M-1}\int_m^{m+1}\exp\big(-\kappa\Ein(x)-ux\big)\,dx
+\int_M^\infty\exp\big(-\kappa\Ein(x)-ux\big)\,dx\\
&=\sum_{m=0}^{M-1}e^{-mu}
\int_0^1\exp\big(-\kappa\Ein(m+z)-uz\big)\,dz
+\int_M^\infty\exp\big(-\kappa\Ein(x)-ux\big)\,dx
\end{align*}
for some $M\in\Z_{>0}$.
For the final term, we use the crude estimate $\Ein(x)\ge\Ein(M)$, so that
$$
\int_M^\infty\exp\big(-\kappa\Ein(x)-ux\big)\,dx
=\Theta([0,1])u^{-1}e^{-\kappa\Ein(M)-Mu}.
$$

Using the series in \eqref{eqn:Edef}, for any positive integer $N$ and
any $z\in[0,1]$, we have
\begin{equation}\label{eqn:Eseries}
\Ein(z)=\sum_{n=1}^{N-1}\frac{(-1)^{n-1}}{n\cdot n!}z^n
+\Theta([0,1])\frac{(-1)^{N-1}}{N\cdot N!}z^N.
\end{equation}
On the other hand, differentiating the integral representation in
\eqref{eqn:Edef}, we have
\begin{equation}\label{eqn:Ederiv}
\Ein^{(k)}(x)=(-1)^{k-1}(k-1)!x^{-k}\left(1-e^{-x}
\sum_{n=0}^{k-1}\frac{x^n}{n!}\right),
\end{equation}
so that
$$
0\le(-1)^{k-1}\frac{\Ein^{(k)}(x)}{k!}=
\frac1ke^{-x}\sum_{n=0}^\infty\frac{x^n}{(n+k)!}
\le\frac1{k\cdot k!}.
$$
Comparing series term by term, we see that for any fixed $x>0$,
the sequence $\Ein^{(k)}(x)/k!$ is alternating and decreasing in
magnitude. Hence,
$$
\Ein(m+z)=\sum_{n=0}^{N-1}\frac{\Ein^{(n)}(m)}{n!}z^n
+\Theta([0,1])\frac{\Ein^{(N)}(m)}{N!}z^N
\quad\mbox{for }z\in[0,1].
$$

For each $m\in\{0,1,\ldots,M-1\}$,
we use \eqref{eqn:Eseries} and \eqref{eqn:Ederiv} to compute
an expansion for $\Ein(m+z)$ for $z\in[0,1]$.
(The constant terms $\Ein(m)$ are computed recursively using the
expansion for $m-1$.)
Applying the algorithms from \S\ref{sec:powerseries},
we obtain an expansion of the form
$$
\exp\big(-\kappa\Ein(m+z)\big)=\sum_{n=0}^Na_{m,n}(z)z^n
\quad\mbox{for }z\in[0,1].
$$
Thus,
$$
p(u)=\sum_{m=0}^{M-1}e^{-mu}
\sum_{n=0}^N\int_0^1 a_{m,n}(z)z^ne^{-uz}\,dz
+\Theta([0,1])u^{-1}e^{-\kappa\Ein(M)-Mu}.
$$

Now fix $u_0\ge 2$ and replace $u$ in the above by $u_0+u$ for some
$u\in[0,1]$. Then
$$
e^{-(u_0+u)z}=e^{-u_0z}\sum_{r=0}^Nb_r(z,u)z^ru^r,
$$
where
$$
b_r(z,u)=\begin{cases}
\frac{(-1)^r}{r!}&\mbox{if }r<N,\\
\Theta([0,1])\frac{(-1)^N}{N!}&\mbox{if }r=N.
\end{cases}
$$
Thus,
\begin{equation}\label{eqn:pexp1}
\begin{aligned}
p(u_0+u)&=\sum_{m=0}^{M-1}\sum_{r=0}^N
c_{u_0,m,r}(u)e^{-mu}u^r
+\Theta([0,1])u_0^{-1}e^{-\kappa\Ein(M)-Mu_0}\\
&=
\sum_{r=0}^N\sum_{m=0}^{M-1}
c_{u_0,m,r}(u)\left(
\sum_{n=0}^{N-r-1}\frac{(-m)^n}{n!}u^{n+r}
+\Theta([0,1])\frac{(-m)^{N-r}}{(N-r)!}u^N\right)\\
&\hspace{1cm}+\Theta([0,1])u_0^{-1}e^{-\kappa\Ein(M)-Mu_0},
\end{aligned}
\end{equation}
where
\begin{align*}
c_{u_0,m,r}(u)&=e^{-mu_0}
\sum_{n=0}^N\int_0^1
a_{m,n}(z)b_r(z,u)z^{n+r}e^{-u_0z}\,dz\\
&=e^{-mu_0}
\sum_{n=0}^N\Theta(I_{m,n,r})
\int_0^1z^{n+r}e^{-u_0z}\,dz,
\end{align*}
and $I_{m,n,r}$ is a bounding interval for $a_{m,n}b_r$.
To compute the integral on the last line accurately,
we rely on {\tt Arb}'s routines for the confluent hypergeometric
function $M$, via the identity
$$
\int_0^1z^ne^{-u_0z}\,dz
=\frac1{n+1}M(n+1,n+2,-u_0).
$$
This yields an expansion of the form
$$
p(u_0+u)=\sum_{n=0}^Np_{u_0,n}(u)u^n
\quad\mbox{for }u\in[0,1].
$$

For small $u_0$, the worst case error in \eqref{eqn:pexp1} occurs when
$m\approx N/u_0$, and is roughly of size
$u_0^{-N}$. If we aim for $B$ bits of precision for all $u_0\ge2$,
then a sensible choice of parameters is $N=B$,
$M=\lceil\frac{B\log2}{u_0}\rceil$.

\subsubsection{Computing $q$}
Since we have restricted to positive integer values of $\kappa$,
it turns out that $q$ is a polynomial of degree $2\kappa-1$.
If we put $q(u)=\sum_{n=0}^{2\kappa-1}a_nu^n$ then, solving
\eqref{eqn:qdef}, we find that the coefficients $a_n$ are given 
recursively by
$$
a_{2\kappa-1}=1,
\quad
a_r=-\frac{\kappa}{2\kappa-1-r}\sum_{n=r+1}^{2\kappa-1}{{n}\choose{r}}a_n
\quad\mbox{for }r<2\kappa-1.
$$
Once the $a_n$ have been computed, we work out the coefficients
of $q(u_0+u)$ for integers $u_0\ge1$ by polynomial arithmetic.
In order to limit the precision loss, we perform these computations in
rational arithmetic first before converting to floating point intervals.

\subsubsection{Computing $\sigma$}
Let $u_0\in\Z_{\ge0}$.  We wish to compute an expansion of the form
$$
\sigma(u_0+u)=\sum_{n=0}^N\sigma_{u_0,n}(u)u^n
\quad\mbox{for }u\in[0,1].
$$
For $u_0=0$ we have
$$
\sigma(u)=A^{-1}u^\kappa
\quad\mbox{for }u\in[0,1],
$$
where $A=\kappa!(2e^\gamma)^\kappa$,
which is an expansion of the above type provided that
$N\ge\kappa$. For $u_0=1$ we similarly have
$$
\sigma(1+u)=
A^{-1}(1+u)^\kappa=A^{-1}\sum_{n=0}^\kappa
{{\kappa}\choose{n}}u^n
\quad\mbox{for }u\in[0,1].
$$

Next, suppose that an expansion of the desired type is known for
$u_0-2$ and $u_0-1$. Solving \eqref{eqn:sigmadef}, we have
$$
\sigma(u_0+u)=\left(1+\frac{u}{u_0}\right)^\kappa\left[
\sigma(u_0)
-\frac{\kappa}{u_0}\int_0^u\left(1+\frac{t}{u_0}\right)^{-\kappa-1}
\sigma(u_0-2+t)\,dt\right].
$$
Using the expansion around $u_0-1$ to estimate $\sigma(u_0)$
and applying the formulas and algorithms from
\S\ref{sec:powerseries} in a straightforward way,
we obtain an expansion for $\sigma(u_0+u)$.

\subsubsection{Computing $1/\sigma$}
For $u_0\ge3$, we use \eqref{eqn:reciprocal}, together with the fact
that $\sigma$ is monotonic \cite[(14.6)]{dh-book}
to obtain a bounding interval, to compute
an expansion for $1/\sigma(u_0+u)$. For $u_0\in\{1,2\}$, this method
gives a poor approximation when $u$ is close to $1$ since, in each case,
the analytic function that agrees with $\sigma(u_0+u)$ on $[0,1]$ has
a zero at or just to the left of $u=-1$. This can be mitigated in
various ways, e.g.\ by using expansions around $u_0=3/2$ or $5/2$ instead.
In order to keep our later algorithms as simple and uniform as possible,
we use the following \emph{ad hoc} methods to obtain better polynomial
approximations of $1/\sigma(u_0+u)$ for $u_0\in\{1,2\}$.

First, for $u_0=1$ and $u\in[0,1]$, we have
\begin{align*}
(1+u)^{-\kappa}&=2^{-\kappa}\bigl(1-\tfrac{1-u}2\bigr)^{-\kappa}
=\sum_{n=0}^\infty2^{-\kappa-n}{{\kappa+n-1}\choose{n}}(1-u)^n.
\end{align*}
We truncate the series at $N$ and express it in the form
$\sum_{n=0}^Nc_nu^n$ by polynomial arithmetic. Since the series
has positive terms, the greatest error occurs at $u=0$. Hence, to
account for the error, it suffices to replace the constant term $c_0$
by $\Theta([c_0,1])$.

Turning to $u_0=2$, we solve \eqref{eqn:sigmadef} to find,
for $u\in[0,1]$,
\begin{equation}\label{eqn:sigma3-u}
\begin{split}
A\sigma(3-u)=~&
(3-u)^\kappa\left(1-\kappa\log\bigl(\tfrac{3-u}2\bigr)\right)
+\kappa\sum_{n=1}^\kappa\frac{(1-u)^n(3-u)^{\kappa-n}}{n}\\
=~&(3-u)^\kappa\left(1-\kappa\log\tfrac32
+\kappa\sum_{n=1}^\infty\frac{u^n}{n3^n}\right)
+\kappa\sum_{n=1}^\kappa\frac{(1-u)^n(3-u)^{\kappa-n}}{n}\\
=~&(3-u)^\kappa\left(1-\kappa\log\tfrac32
+\kappa\sum_{n=1}^N\frac{u^n}{n3^n}
+\frac{\kappa\Theta([0,1])}{2(N+1)3^N}u^{N+1}\right)\\
&\quad 
+\kappa\sum_{n=1}^\kappa\frac{(1-u)^n(3-u)^{\kappa-n}}{n}.
\end{split}
\end{equation}
We compute this series using the arithmetic of
\S\ref{sec:powerseries} and invert it to obtain a
degree-$(N+1)$ approximation for $1/\sigma(3-u)$.  Since the error in
\eqref{eqn:sigma3-u} is concentrated in the highest-degree term, the same
is true of our approximation to the inverse. Hence, in the degree-$N$
polynomial part we may replace $u$ by $1-u$ without excessive precision
loss. (There is some loss due to cancellation between terms, but that
can be compensated for by increasing the working precision.)
In the final term, we simply replace $u^{N+1}$ by $\Theta([0,1])$
and incorporate it into the constant term.

Finally, for $u_0=0$, $1/\sigma(u_0+u)=Au^{-\kappa}$ has an unavoidable
singularity at $u=0$. However, the computation of $1/\sigma$
in this region is only needed in order to compute the integrals
$\int_w^{v-1}F(u)\,du$ and $\int_w^{v-1}\frac{F(u)}{v-u}\,du$, and we
can evaluate the contribution from $u\in[w,1]$ by elementary means,
as detailed below.

\subsubsection{Computing $\widetilde{\Pi}$ and $\widetilde{\Xi}$}
For integers $u_0\ge2$, we compute expansions for
$\widetilde{\Pi}(u_0+u)$ and $\widetilde{\Xi}(u_0+u)$
via \eqref{eqn:pixidef}, using our expansions for $p$, $q$,
$1/\sigma$ and a straightforward application of the operations from
\S\ref{sec:powerseries}.

\subsubsection{Computing $\alpha$ and $\beta$}
With these preliminaries in place, we can now proceed with the
computation of $\alpha$ and $\beta$. First, eliminating the integral
terms in \eqref{eqn:alphabetadef}, we see that $\alpha$ is a root of
the function
$$
l(u)=
\bigl(\widetilde{\Pi}(u)-2\bigr)q(u-1)+\widetilde{\Xi}(u)p(u-1).
$$
In \cite{dh-parameters} it was shown that $l$ has a root satisfying
$\rho+1<\alpha<\rho+O(1)$, where $\rho$ denotes the greatest
real zero of $q$. In every case that we examined we found that
$l(\rho+1)l(\rho+2)<0$, so there is a solution $\alpha\in(\rho+1,\rho+2)$.
(In \cite[Prop.~17.3]{dh-book} it is proven that there is a solution
with $\alpha<3.75\kappa$, so it is not necessary to assume that
$\alpha<\rho+2$, but we did so for efficiency.)  To locate it
precisely, we compute expansions for $l(u_0+u)$ as above and apply
bisection\footnote{Since we store all numbers as intervals, if it
happens that $\alpha$ is very close to an integer then it might not be
possible to decide which value of $u_0$ to use at some point during the
bisection algorithm.  One could handle this by taking the union of the
intervals arising from all relevant values of $u_0$. However, we expect
that this does not occur generically; in our implementation, we simply
output an error message and exit when it occurs, and we never observed
it in practice.  The same issue arises at various points for computing
$\beta$, as well as the parameters $v$ and $w$ below. We handled these
in a similar fashion and proceed without further comment.}, beginning
with the interval $(\rho+1,\rho+2)$.

Once $\alpha$ has been found, $\beta<\alpha-1$ is uniquely determined by
\eqref{eqn:alphabetadef}, since $\frac{t^{\kappa-1}}{\sigma(t-1)}>0$
and $\frac{\widetilde{\Xi}(\alpha)}{q(\alpha-1)}>0$
\cite[(17.9)]{dh-book}.  By \cite[Prop.~7.3]{dhr2} and
\cite[Table 17.1]{dh-book}, we have $\beta>2\kappa$ for all
integral $\kappa\ge2$; to locate it we compute expansions for
$\kappa\int_{u_0+u}^{\alpha-1}\frac{t^{\kappa-1}}{\sigma(t-1)}\,dt$ and
again apply bisection, beginning with the interval $(2\kappa,\alpha-1)$.

Our computed values of $\alpha$ and $\beta$ for $\kappa\le50$ are
displayed to $20$ decimal place accuracy at \cite{sourcecode}.

\subsubsection{Computing $F$ and $f$}
We compute expansions for $F(u_0+u)$ and $f(u_0+u)$ by a similar
strategy to that for $\sigma$. Note that for
$u_0\ge\lfloor\alpha\rfloor$, $F(u_0+u)$ changes behaviour at
$u=\{\alpha\}=\alpha-\lfloor\alpha\rfloor$, and similarly for $f(u_0+u)$
when $u_0\ge\lfloor\alpha\rfloor+1$. Hence, for each $u_0$, we compute
two approximations for $F(u_0+u)$ and $f(u_0+u)$, one valid for
$u\in[0,\{\alpha\}]$ and one for $u\in[\{\alpha\},1]$. Solving
\eqref{eqn:Ffdef}, we have
$$
F(u_0+u)=\left(1+\frac{u}{u_0}\right)^{-\kappa}\left[c+\frac{\kappa}{u_0}
\int_0^u\left(1+\frac{t}{u_0}\right)^{\kappa-1}f(u_0-1+t)\,dt\right],
$$
where $c$ is chosen to ensure continuity at $u=0$ or $u=\{\alpha\}$,
and similarly with the roles of $F$ and $f$ reversed.
Applying this iteratively yields the desired expansions.

\subsection{Optimising $R(v,w)$}
Differentiating \eqref{eqn:Rdef}, we find that
\begin{equation}\label{eqn:dRdw}
f(v)(v-w)^2\frac{\partial R}{\partial w}=
hvf(v)-\kappa\int_w^{v-1}F(u)\,du.
\end{equation}
Further, we have
\begin{align*}
\frac{d}{dv}&\left(
hvf(v)-\kappa\int_{\beta-1}^{v-1}F(u)\,du\right)
=h(f(v)+vf'(v))-\kappa F(v-1)\\
&=(h-\kappa)F(v-1)+h(1-\kappa^{-1})vf'(v)>0
\quad\mbox{for }v>\beta,
\end{align*}
and since
$[hvf(v)-\kappa\int_{\beta-1}^{v-1}F(u)\,du]_{v=\beta}=0$,
it follows that 
$hvf(v)-\kappa\int_{\beta-1}^{v-1}F(u)\,du>0$ for all $v>\beta$.

Since $F(u)>1$ for all $u$, \eqref{eqn:dRdw} is a strictly increasing
function of $w>0$. It tends to $-\infty$ as $w\to0^+$ and, by the above,
is positive at $w=\beta-1$; hence, for each fixed $v>\beta$,
there is a unique $w=w(v)\in(0,\beta-1)$ at which $R(v,w)$ is
minimal.  Given a value of $v$, we compute $w(v)$ as follows.
First we compute
$c(v)=hvf(v)-\kappa\int_2^{v-1}F(u)\,du$. If $c(v)\ge0$ then
$w(v)\le2$, and we have
$$
c(v)=\kappa\int_{w(v)}^2F(u)\,du=A\kappa\int_{w(v)}^2u^{-\kappa}\,du
=\frac{A\kappa}{\kappa-1}\bigl({w(v)}^{1-\kappa}-2^{1-\kappa}\bigr),
$$
which we then solve for $w(v)$.  Otherwise, $w(v)\in(2,\beta-1)$, and we
find the root of \eqref{eqn:dRdw} by bisection.

Replacing $w$ by $w(v)$ in \eqref{eqn:Rdef} and applying
\eqref{eqn:dRdw}, we have
$$
R(v,w(v))=\frac{\kappa}{f(v)}\int_{w(v)}^{v-1}\frac{F(u)}{v-u}\,du.
$$
To compute the integral, we split it over the intervals
$[u_0,u_0+\{\alpha\}]$ and $[u_0+\{\alpha\},u_0+1]$ for each relevant
integer $u_0$, taking the appropriate subinterval around the endpoints
$w(v)$ and $v-1$.  For each interval $[u_1,u_2]\subseteq[u_0,u_0+1]$
with $u_0\ge2$ we use the approximation
$$
\frac1{v-u}=\frac1{(v-u_0)-(u-u_0)}=
\sum_{n=0}^{N-1}(v-u_0)^{-n-1}(u-u_0)^n
+\frac{(v-u_0)^{-N}(u-u_0)^N}{v-\Theta([u_1,u_2])}.
$$
Note that we always have $\frac{u_2-u_0}{v-u_0}\le\frac12$, so this
is accurate to at least $N$ bits.
When $w(v)<2$, we compute the contribution from $[w(v),2]$ directly,
via
\begin{align*}
\int_{w(v)}^2\frac{du}{u^\kappa(v-u)}=
\left[v^{-\kappa}\log\frac1{v-u}-
\sum_{n=1}^{\kappa-1}\frac{v^{n-\kappa}u^{-n}}{n}\right]_{u=w(v)}^{u=2}.
\end{align*}

Thus, we have reduced the problem to one of minimising the single-variable
function $R(v)=R(v,w(v))$.  Presumably, $R(v)$ has a unique local (and
global) minimum on $(\beta,\infty)$; although we are not aware of a
proof of this, we found no violations of it in practice.  Empirically,
the optimal $v$ is always larger than $\alpha+3$. A good upper bound that
is uniform in all parameters is not as easy to describe, but by trial
and error we found that $v<200$ for every $\kappa\le h\le50$. Hence, to
compute the table in \cite{sourcecode}, we first worked out the expansions
for $F$ and $f$ for all $u_0<200$. We then used a simple linear search
through the points $v=\lceil\alpha\rceil+n$ for $n=1,2,3\ldots$ to find
one at which $R(v)<\min(R(v-1),R(v+1))$.  Finally, we zoomed in on the
minimum by a simple bisection algorithm.

\subsection{Working precision}
We conclude this section with a few words on numerical precision. Since
our method for computing $f$ and $F$ is iterative, the total precision
loss accumulates as $u_0$ increases. This mainly affects large values
of $\kappa$, for which the optimal $v$ can be large. We
counteracted this effect by increasing the working precision with
$\kappa$. Empirically we found that using $12(\kappa+10)$ bits of
precision was enough to determine the minimum value of $R(v,w)$ to at
least $20$ significant (decimal) digits, for all $\kappa\le h\le50$.
The demand for high precision is the main obstacle to extending our
computations to larger values of $\kappa$, though we expect that
$\kappa$ in the hundreds would still be feasible.

%%% AUTHOR:
%%% Bibliography goes here. Note that the arXiv cannot process bibtex
%%% or biber bibliographies.  Example of acceptable bibliograpy format:

\bibliographystyle{amsplain}

%%% AUTHOR: Include a short description of each author following the
%%% structure below. Use the same short tags used previously.  
%%% Use \imageat{} and \imagedot{} instead of "@" and "." in
%%% email addresses-this replaces the symbols with graphics to avoid 
%%% e-mail address harvesting from the .pdf file
\begin{dajauthors}
\begin{authorinfo}[ARB]
  Andrew R. Booker\\
Doctor\\
University of Bristol\\ 
Bristol, UK\\
andrew.booker\imageat{}bristol\imagedot{}ac\imagedot{}uk\\
\end{authorinfo}

\begin{authorinfo}[ARB]
  Tim D. Browning\\
Professor\\
University of Bristol\\ 
Bristol, UK\\
t.d.browning\imageat{}bristol\imagedot{}ac\imagedot{}uk\\
\end{authorinfo}

\end{dajauthors}

\end{document}